\documentclass{amsart}
\usepackage{amsmath}
\usepackage[psamsfonts]{amssymb}
\usepackage[dvips]{graphics}
\usepackage[all]{xy}
\newtheorem{Theorem}{Theorem}[section]
\newtheorem{Lemma}[Theorem]{Lemma}

\theoremstyle{definition}

\theoremstyle{remark}
\newtheorem{Remark}[Theorem]{Remark}

\def\labelenumi{\rm ({\makebox[0.8em][c]{\theenumi}})}
\def\theenumi{\arabic{enumi}}
\def\labelenumii{\rm {\makebox[0.8em][c]{(\theenumii)}}}
\def\theenumii{\roman{enumii}}

\renewcommand{\thefigure}
{\arabic{section}.\arabic{figure}}
\makeatletter
\@addtoreset{figure}{section}
\def\@thmcountersep{-}
\makeatother

\numberwithin{equation}{section}

\begin{document} 

\title[$C_{n}$-moves and the difference of Jones polynomials for links]{$C_{n}$-moves and the difference of Jones polynomials for links}

\author{Ryo Nikkuni}
\address{Department of Mathematics, School of Arts and Sciences, Tokyo Woman's Christian University, 2-6-1 Zempukuji, Suginami-ku, Tokyo 167-8585, Japan}
\email{nick@lab.twcu.ac.jp}
\thanks{The author was supported by JSPS KAKENHI Grant Number 15K04881.}



\date{}



\begin{abstract}
The Jones polynomial $V_{L}(t)$ for an oriented link $L$ is a one-variable Laurent polynomial link invariant discovered by Jones. For any integer $n\ge 3$, we show that: (1) the difference of Jones polynomials for two oriented links which are $C_{n}$-equivalent is divisible by $\left(t-1\right)^{n}\left(t^{2}+t+1\right)\left(t^{2}+1\right)$, and (2) there exists a pair of two oriented knots which are $C_{n}$-equivalent such that the difference of the Jones polynomials for them equals $\left(t-1\right)^{n}\left(t^{2}+t+1\right)\left(t^{2}+1\right)$. 
\end{abstract}

\maketitle

\section{Introduction} 

The {\it Jones polynomial} $V_{L}(t)\in {\mathbb Z}\left[t^{\pm 1/2}\right]$ is an integral Laurent polynomial link invariant for an oriented link $L$ defined by the following formulae: 
\begin{eqnarray*}
V_{O}(t) &=&1,\\
t^{-1}V_{L_{+}}(t)-tV_{L_{-}}(t) &=& \left(t^{\frac{1}{2}}-t^{-\frac{1}{2}}\right)V_{L_{0}}(t), 
\end{eqnarray*}
where $O$ denotes the trivial knot and $L_{+}$, $L_{-}$ and $L_{0}$ are oriented links which are identical except inside the depicted regions as illustrated in Fig. \ref{skein} \cite{Jones87}. The triple of oriented links $\left(L_{+},L_{-},L_{0}\right)$ is called a {\it skein triple}. Jones also showed the following property of the Jones polynomials for oriented knots. 

\begin{Theorem}\label{jw}{\rm (Jones \cite[Proposition 12.5]{Jones87})} 
For any two oriented knots $J$ and $K$, $V_{J}(t)-V_{K}(t)$ is divisible by $\left(t-1\right)^{2}\left(t^{2}+t+1\right)$.
\end{Theorem}

On the basis of Theorem \ref{jw}, for an oriented knot $K$, Jones called the polynomial $W_{K}(t) = \left\{1-V_{K}(t)\right\} / \left(t-1\right)^{2}\left(t^{2}+t+1\right)$ a {\it simplified polynomial} and made a table of the simplified polynomials for knots up to $10$ crossings \cite{Jones87}. In particular, if $K$ is the right-handed trefoil knot then $W_{K}(t) = 1$. So the polynomial $\left(t-1\right)^{2}\left(t^{2}+t+1\right)$ is maximal as a divisor of the difference of Jones polynomials of any pair of two oriented knots.

\begin{figure}[htbp]
      \begin{center}
\scalebox{0.4}{\includegraphics*{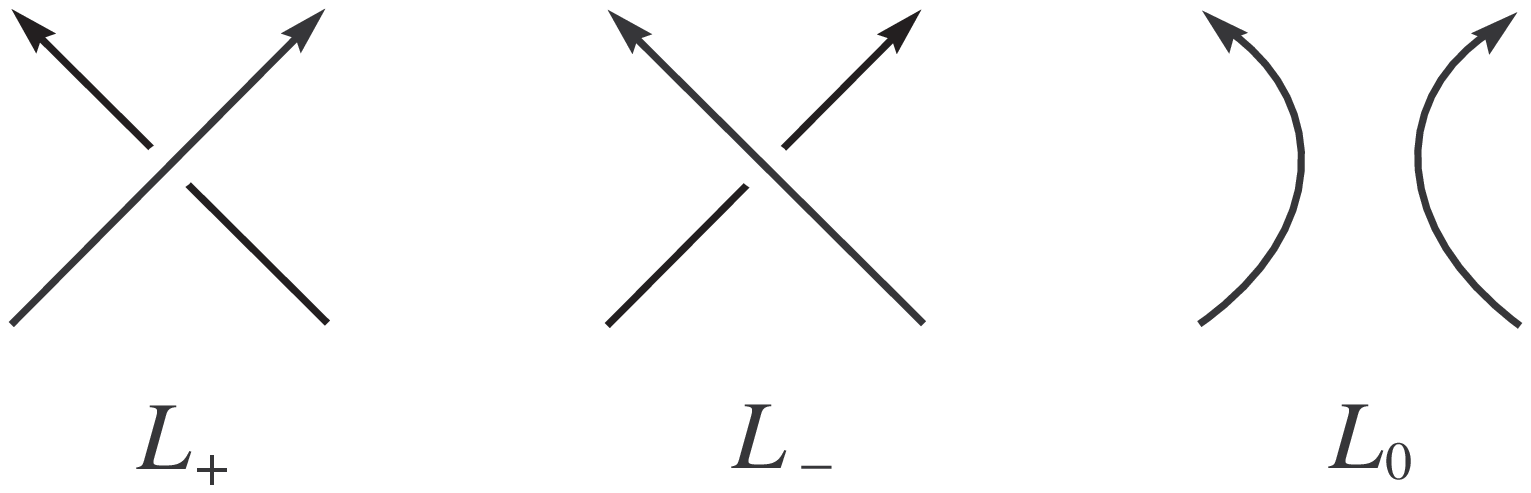}}
      \end{center}
   \caption{Skein triple $\left(L_{+},L_{-},L_{0}\right)$}
  \label{skein}
\end{figure} 

Our purpose in this paper is to examine the difference of Jones polynomials for two oriented links which are {\it $C_{n}$-equivalent}, where a $C_{n}$-equivalence is an equivalence relation on oriented links introduced by Habiro \cite{Habiro94} and Gusarov \cite{Gusarov00} independently as follows. For a positive integer $n$, a {\it $C_{n}$-move} is a local move on oriented links as illustrated in Fig. \ref{Cnmove} if $n\ge 2$, and a $C_{1}$-move is a crossing change. Two oriented links are said to be {\it $C_{n}$-equivalent} if they are transformed into each other by $C_{n}$-moves and ambient isotopies. By the definition of a $C_{n}$-move, it is easy to see that a $C_{n}$-equivalence implies a $C_{n-1}$-equivalence. Note that a $C_{2}$-move equals a {\it delta move} introduced by Matveev \cite{Matveev87} and Murakami-Nakanishi \cite{MN89} independently as illustrated in Fig. \ref{C2C3} (1), and a $C_{3}$-move equals a {\it clasp-pass move} introduced by Habiro \cite{Habiro93} as illustrated in Fig. \ref{C2C3} (2). A $C_{n}$-move is closely related to the {\it Vassiliev invariants} of oriented links \cite{Vass90}, \cite{BL93}, \cite{Natan95}, \cite{Stanford96}. It is known that if two oriented links are $C_{n}$-equivalent then they have the same Vassiliev invariants of order $\le n-1$, and specially for oriented knots, the converse is also true \cite{Habiro00}, \cite{Gusarov00}.

\begin{figure}[htbp]
      \begin{center}
\scalebox{0.475}{\includegraphics*{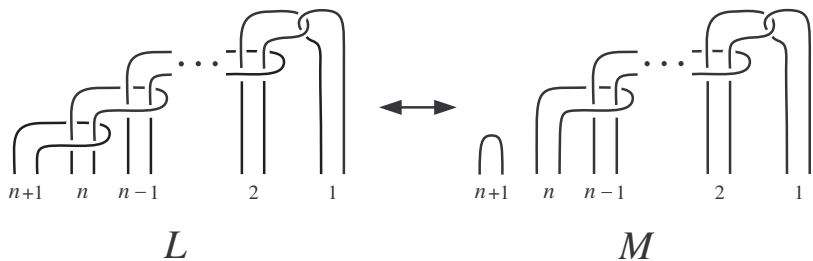}}
      \end{center}
   \caption{$C_{n}$-move ($n\ge 2$)}
  \label{Cnmove}
\end{figure} 
\begin{figure}[htbp]
      \begin{center}
\scalebox{0.45}{\includegraphics*{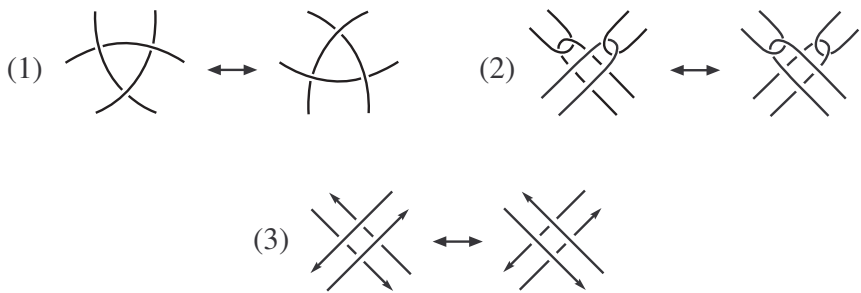}}
      \end{center}
   \caption{(1) Delta move, (2) Clasp-pass move, (3) Pass move}
  \label{C2C3}
\end{figure} 

Now let us generalize Theorem \ref{jw} to oriented links which are $C_{n}$-equivalent. 

\begin{Theorem}\label{divisible}
\begin{enumerate}
\item If two oriented links $L$ and $M$ are $C_{2}$-equivalent, then $V_{L}(t)-V_{M}(t)$ is divisible by $\left(t-1\right)^{2}\left(t^{2}+t+1\right)$.
\item For any integer $n\ge 3$, if two oriented links $L$ and $M$ are $C_{n}$-equivalent, then $V_{L}(t)-V_{M}(t)$ is divisible by $\left(t-1\right)^{n}\left(t^{2}+t+1\right)\left(t^{2}+1\right)$. 
\end{enumerate}
\end{Theorem}

We remark that Theorem \ref{divisible} (1) was also observed in \cite[Theorem 2]{Ganzell14} for oriented knots by using the Kauffman bracket. Since any two oriented knots are $C_{2}$-equivalent \cite{MN89}, Theorem \ref{jw} is deduced from Theorem \ref{divisible} (1). 

In the case of $n\ge 3$, we show the maximality of $\left(t-1\right)^{n}\left(t^{2}+t+1\right)\left(t^{2}+1\right)$ as a divisor of the difference of Jones polynomials for oriented links which are $C_{n}$-equivalent as follows. Let $J_{n}$ and $K_{n}$ be two oriented knots as illustrated in Fig. \ref{Cnex}. Note that $J_{n}$ and $K_{n}$ are transformed into each other by a single $C_{n}$-move, see Fig. \ref{Cnex2}. Then we have the following. 

\begin{Theorem}\label{main_jones}
\begin{eqnarray*}
V_{J_{n}}(t) - V_{K_{n}}(t) 
= (-1)^{n+1}(t-1)^{n}\left(t^{2}+t+1\right)\left(t^{2}+1\right). 
\end{eqnarray*}
\end{Theorem}

\begin{figure}[htbp]
      \begin{center}
\scalebox{0.4}{\includegraphics*{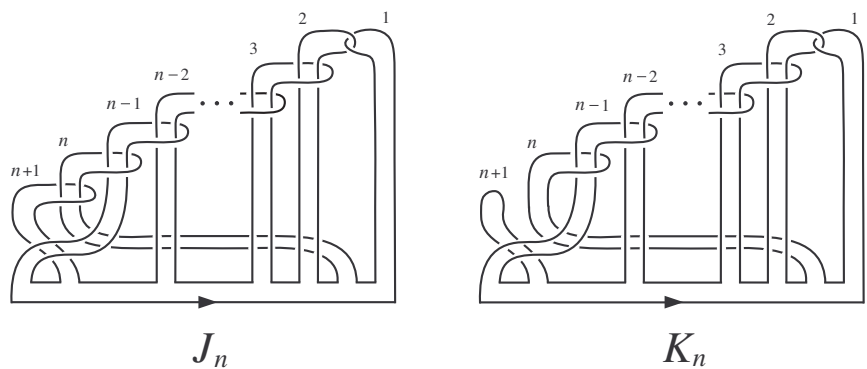}}
      \end{center}
   \caption{Oriented knots $J_{n}$ and $K_{n}$ ($n\ge 3$)}
  \label{Cnex}
\end{figure} 
\begin{figure}[htbp]
      \begin{center}
\scalebox{0.375}{\includegraphics*{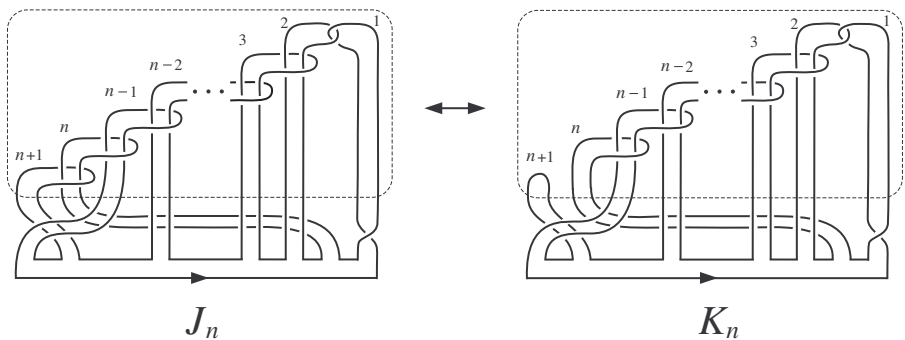}}
      \end{center}
   \caption{$J_{n}$ and $K_{n}$ are transformed into each other by a single $C_{n}$-move}
  \label{Cnex2}
\end{figure} 

In section $2$, we prove Theorem \ref{divisible} and give its applications to the study of the difference of Vassiliev invariants of order $\le n$ for two oriented links which are $C_{n}$-equivalent. In section $3$, we prove Theorem \ref{main_jones} without knowing $V_{J_{n}}(t)$ and $V_{K_{n}}(t)$ individually by applying Kanenobu's formula for the difference of Jones polynomials for two oriented knots which are transformed into each other by a single $C_{n}$-move (Lemma \ref{lem1}) and a $C_{n}$-move which does not change the knot type (Lemma \ref{lem2}).

\section{Proof of Theorem \ref{divisible}} 

We recall the following results about the special values of the Jones polynomial. Here an $r$-component oriented link $L$ is said to be {\it proper} if ${\rm lk}\left(K,L\setminus K\right)\equiv 0\pmod{2}$ for each component $K$ of $L$, where ${\rm lk}$ denotes the linking number, and the {\it Arf invariant} is a link invariant introduced in \cite{Robertello} defined for only proper links.

\begin{Lemma}\label{vlem} Let $L$ be an $r$-component oriented link. Then the followings holds. 
\begin{enumerate}
\item {\rm (\cite[(12.1)]{Jones87})} $V_{L}(1) = \left(-2\right)^{r-1}$. 

\item {\rm (\cite[(12.4)]{Jones87})} $\displaystyle V_{L}\left(e^{{2\pi \sqrt{-1}}/3}\right) = \left(-1\right)^{r-1}$. 

\item {\rm (Murakami \cite{Murakami86})} $V_{L}(\sqrt{-1}) = \left(\sqrt{-2}\right)^{r-1}\cdot (-1)^{{\rm Arf}(L)}$ if $L$ is proper, and $0$ if $L$ is nonproper, where ${\rm Arf}$ denotes the {\it Arf invariant}. 
\end{enumerate}
\end{Lemma}

For an oriented link $L$, we denote the $l$-th derivative at $1$ of the Jones polynomial $V_{L}(t)$ by $V_{L}^{(l)}(1)$. It is known that $V_{L}^{(l)}(1)$ is a Vassiliev invariant of order $\le l$ \cite{KN98}. Then we have the following.

\begin{Lemma}\label{mainlem1}
Let $L$ and $M$ be two oriented $r$-component links and $n$ an integer with $n\ge 2$. Then $V_{L}^{(l)}(1) = V_{M}^{(l)}(1)$ for $l=1,2,\ldots, n-1$ if and only if $V_{L}(t)-V_{M}(t)$ is divisible by $\left(t-1\right)^{n}\left(t^{2}+t+1\right)$. 
\end{Lemma}

\begin{proof}
The `if' part is clear because $\left(t-1\right)^{n}$ divides $V_{L}(t)-V_{M}(t)$. We show the `only if' part by the induction on $n$. Assume that $n=2$. By Lemma \ref{vlem} (1) and (2), there exists a polynomal $g(t)\in {\mathbb Z}\left[t^{\pm 1/2}\right]$ such that 
\begin{eqnarray}\label{d1}
V_{L}(t) - V_{M}(t) = \left(t^{3}-1\right)g(t). 
\end{eqnarray}
Then by differentiating both sides in (\ref{d1}), we have 
\begin{eqnarray}\label{d2}
V_{L}^{(1)}(t) - V_{M}^{(1)}(t) = 3t^{2}g(t) + \left(t^{3}-1\right)g^{(1)}(t).  
\end{eqnarray}
Thus by the assumption and (\ref{d2}), we have $g(1)=0$. This implies that $V_{L}(t)-V_{M}(t)$ is divisible by $\left(t-1\right)^{2}\left(t^{2}+t+1\right)$. Next assume that $n\ge 3$ and $V_{L}^{(l)}(1) = V_{M}^{(l)}(1)$ for $l=0,1,\ldots, n-1$. By the induction hypothesis, it follows that there exists a polynomial $h(t)\in\left[t^{\pm 1/2}\right]$ such that 
\begin{eqnarray}\label{d3}
V_{L}(t) - V_{M}(t) = \left(t-1\right)^{n-1}\left(t^{2}+t+1\right)h(t). 
\end{eqnarray}
Let us denote $\left(t^{2}+t+1\right)h(t)$ by $\tilde{h}(t)$. Then by (\ref{d3}) and the assumption, we have 
\begin{eqnarray}\label{d4}
0 = V_{L}^{(n-1)}(1) - V_{M}^{(n-1)}(1) = (n-1)!\ \tilde{h}(1). 
\end{eqnarray}
Thus we have $0 = \tilde{h}(1) = 3h(1)$, namely $h(1) = 0$. This implies that $t-1$ divides $h(t)$, therefore we have the desired conclusion. 
\end{proof}

\begin{Remark}
{\rm For an $r$-component oriented link $L$, it is known that 
\begin{eqnarray*}
V_{L}^{(1)}(1) = -3(-2)^{r-2}{\rm Lk}(L)
\end{eqnarray*}
 if $r\ge 2$ and $0$ if $r=1$, where ${\rm Lk}$ denotes the total linking number, that is the summation of all pairwise linking numbers of $L$ \cite[(12.2)]{Jones87}. Thus Lemma \ref{mainlem1} implies Theorem \ref{jw} in case $r=1$, and ${\rm Lk}(L) = {\rm Lk}(M)$ if and only if $V_{L}(t)-V_{M}(t)$ is divisible by $\left(t-1\right)^{2}\left(t^{2}+t+1\right)$ in case $r\ge 2$. 
}
\end{Remark}

\begin{Lemma}\label{mainlem2}
For an integer $n\ge 3$, if two oriented links $L$ and $M$ are $C_{n}$-equivalent, then $V_{L}(t)-V_{M}(t)$ is divisible by $t^{2} + 1$. 
\end{Lemma}

\begin{proof}
Let $L$ and $M$ be two $r$-component oriented links which are $C_{n}$-equivalent. Then $L$ and $M$ are $C_{3}$-equivalent. Note that a $C_{3}$-move $=$ a clasp-pass move can be realized by a single {\it pass move} \cite{Kauffman83} as illustrated in Fig. \ref{C2C3} (3), and a pass move does not change the Arf invariant of a proper link \cite[Appendix]{MN89}. If $L$ is proper, then $M$ is also proper because $L$ and $M$ also are $C_{2}$-equivalent and a $C_{2}$-move does not change the pairwise linking numbers. Then by Lemma \ref{vlem} (3), we have $V_{L}\left(\sqrt{-1}\right) = V_{M}\left(\sqrt{-1}\right)$. If $L$ is nonproper, then $M$ is also nonproper and by Lemma \ref{vlem} (3), we have $V_{L}\left(\sqrt{-1}\right) = 0 = V_{M}\left(\sqrt{-1}\right)$. 
\end{proof}

\begin{proof}[Proof of Theorem \ref{divisible}]
As we mentioned before, if two oriented $r$-component links $L$ and $M$ are $C_{n}$-equivalent then $V_{L}^{(l)}(1) = V_{M}^{(l)}(1)$ for $l=1,2,\ldots, n-1$. By combining this fact with Lemma \ref{mainlem1}, we have (1) in case $n=2$, and by combining this fact with Lemma \ref{mainlem1} and Lemma \ref{mainlem2}, we have (2). 
\end{proof}

As an application, we give alternative short proofs for two theorems shown by H. A. Miyazawa. Note that these theorems were proved by fairly combinatorial argument, that is, by making up a list of oriented $C_{n}$-moves carefully and checking the congruence for each of the cases. First we show the following as a direct consequence of Theorem \ref{divisible} (2).

\begin{Theorem}\label{mv1}{\rm (H. A. Miyazawa \cite[Theorem 1.5]{Miyazawa00})} 
For an integer $n\ge 3$, if two oriented links $L$ and $M$ are $C_{n}$-equivalent, then it follows that 
\begin{eqnarray*}
V_{L}^{(n)}(1) \equiv  V_{M}^{(n)}(1) \pmod{6\cdot n!}. 
\end{eqnarray*}
\end{Theorem}

\begin{proof}
Assume that two oriented links $L$ and $M$ are $C_{n}$-equivalent. Then by Theorem \ref{divisible} (2), there exists a polynomial $f(t)\in {\mathbb Z}\left[t^{\pm 1/2}\right]$ such that 
\begin{eqnarray}\label{mv3}
V_{L}(t) - V_{M}(t) = \left(t-1\right)^{n}\left(t^{2}+t+1\right)\left(t^{2}+1\right)f(t). 
\end{eqnarray}
Let us denote $\left(t^{2}+t+1\right)\left(t^{2}+1\right)f(t)$ by $\tilde{f}(t)$. Then by (\ref{mv3}), we have 
\begin{eqnarray*}
V_{L}^{(n)}(1) - V_{M}^{(n)}(1) = n! \cdot \tilde{f}(1) = n! \cdot 6 f(1). 
\end{eqnarray*}
Thus we have the result. 
\end{proof}

Miyazawa also showed the best possibility of Theorem \ref{mv1} by exhibiting two pairs of two oriented knots which are $C_{n}$-equivalent whose differences of the Jones polynomials do not equal $6\cdot n!$ but the greatest common divisor of them is $6\cdot n!$. The best possibility of Theorem \ref{mv1} may also be given by two oriented knots $J_{n}$ and $K_{n}$ in Theorem \ref{main_jones} whose difference of the Jones polynomials exactly equals $6\cdot n!$. Such an example was also observed by Horiuchi \cite{Horiuchi}. 

On the other hand, the {\it Conway polynomial} $\nabla_{L}(z)\in {\mathbb Z}\left[z\right]$ is an integral polynomial link invariant for an oriented link $L$ defined by the following formulae: 
\begin{eqnarray*}
\nabla_{O}(z) &=&1,\\
\nabla_{L_{+}}(z) - \nabla_{L_{-}}(z) &=& z\nabla_{L_{0}}(z), 
\end{eqnarray*}
where $\left(L_{+},L_{-},L_{0}\right)$ is a skein triple in Fig. \ref{skein} \cite{Conway}. 
Note that 
\begin{eqnarray}\label{det}
V_{L}(-1) = \nabla_{L}\left(-2\sqrt{-1}\right) 
\end{eqnarray}
and the absolute value of $V_{L}(-1)$ is known as the determinant of $L$. We denote the coefficient of $z^{l}$ in $\nabla_{L}(z)$ by $a_{l}(L)$. Then it is known that the Conway polynomial of an $r$-component oriented link $L$ is of the following form 
\begin{eqnarray}\label{conway}
\nabla_{L}(z) = \sum_{i\ge 0}a_{r+2i-1}(L) z^{r+2i-1}. 
\end{eqnarray}

It is known that $a_{l}(L)$ is a Vassiliev invariant of order $\le l$ \cite{Natan95}. Thus if two oriented links $L$ and $M$ are $C_{n}$-equivalent, then $a_{l}(L) = a_{l}(M)$ for $l\le n-1$. In the case of $l=n$, Miyazawa showed the following. Note that in the case of oriented knots, this had been obtained by Ohyama-Ogushi \cite{OO90}. 

\begin{Theorem}\label{mv2}{\rm (H. A. Miyazawa \cite[Theorem 1.3]{Miyazawa00})} 
For an integer $n\ge 3$, if two oriented links $L$ and $M$ are $C_{n}$-equivalent, then it follows that 
\begin{eqnarray*}
a_{n}\left(L\right) \equiv a_{n}\left(M\right) \pmod{2}. 
\end{eqnarray*}
\end{Theorem}

\begin{proof}
Let $L$ and $M$ be two $r$-component oriented links which are $C_{n}$-equivalent. If $n\equiv r\pmod{2}$, then  by (\ref{conway}) we have $a_{n}(L) = a_{n}(M) = 0$. Assume that $n\not\equiv r\pmod{2}$. Then by Theorem \ref{divisible} (2) and (\ref{det}), there exists a polynomial $W(t)\in {\mathbb Z}\left[t^{\pm 1/2}\right]$ such that 
\begin{eqnarray*}
(-1)^{n} \cdot 2^{n+1}\cdot W(-1) &=& V_{L}(-1) - V_{M}(-1) \\
&=& \nabla_{L}\left(-2\sqrt{-1}\right) - \nabla_{M}\left(-2\sqrt{-1}\right) \\
&=& \sum_{i\ge 1}\left\{a_{n+2i-2}(L) - a_{n+2i-2}(M)\right\}\cdot \left(-2\sqrt{-1}\right)^{n+2i-2}.
\end{eqnarray*}
This implies 
\begin{eqnarray*}
0 \equiv \left\{a_{n}(L) - a_{n}(M)\right\}\cdot 2^{n}\pmod{2^{n+1}}
\end{eqnarray*}
and therefore $a_{n}(L) - a_{n}(M)$ must be even. 
\end{proof}

Miyazawa showed that Theorem \ref{mv2} is also best possible. Furthermore, Ohyama-Yamada proved that for an integer $n\ge 2$, if two oriented knots $J$ and $K$ are transformed into each other by a single $C_{2n}$-move then $a_{2n}(J) - a_{2n}(K) = 0,\ \pm 2$ \cite[Theorem 1.3]{OY08}.

\section{Proof of Theorem \ref{main_jones}} 

We show three lemmas needed to prove the Theorem \ref{main_jones}. The first lemma is Kanenobu's formula for the difference of Jones polynomials for two oriented links which are transformed into each other by a single $C_{n}$-move. Let $L$ and $M$ be two oriented links which are transformed into each other by a single $C_{n}$-move as illustrated in Fig. \ref{Cnmove}. Let $c_{j1},c_{j2}\ (j=2,3,\ldots,n)$ and $c_{1}$ be crossings of $L$ as illustrated in Fig. \ref{Cnmove2}. We denote the sign of $c_{1}$ by $\varepsilon_{1}$ and the sign of $c_{j1}$ by $\varepsilon_{j}$ $(j=2,3,\ldots,n)$. Let $L\left[\delta_{2},\delta_{3},\ldots,\delta_{n}\right]$ be the link obtained from $L$ by smoothing the crossing $c_{1}$, smoothing the crossing $c_{j1}$ if $\delta_{j} = 1$, and changing the crossing $c_{j1}$ and smoothing the crossing $c_{j2}$ if $\delta_{j} = -1$ $(j=2,3,\ldots,n)$. Then the following formula holds.

\begin{Lemma}\label{lem1} {\rm (Kanenobu \cite[(4.10)]{Kanenobu04})} 
\begin{eqnarray*}
&& V_{L}(t) - V_{M}(t) \\
&=& 
\left(\prod_{i=1}^{n}\varepsilon_{i}\right)t^{\sum_{i=1}^{n}\varepsilon_{i}-\frac{n}{2}}\left(t-1\right)^{n}
\sum_{\delta_{2},\delta_{3},\ldots,\delta_{n}=\pm 1}\left(\prod_{j=2}^{n}\delta_{j}\right)V_{L\left[\delta_{2},\delta_{3},\ldots,\delta_{n}\right]}(t). 
\end{eqnarray*}
\end{Lemma}

\begin{figure}[htbp]
      \begin{center}
\scalebox{0.45}{\includegraphics*{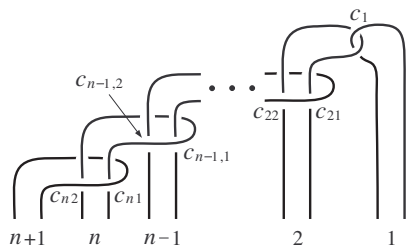}}
      \end{center}
   \caption{Crossings $c_{1}$, $c_{j1}$ and $c_{j2}$ of $L$ ($j=2,3,\ldots,n$)}
  \label{Cnmove2}
\end{figure} 

Next we show the second lemma. For an integer $n\ge 3$, let $L'_{n}$ and $M'_{n}$ be two links as illustrated in Fig. \ref{Cnambi}, where $T$ is an arbitrary $2$-string tangle which are same for both links. Note that $L'_{n}$ and $M'_{n}$ are transformed into each other by a single $C_{n}$-move. Then we have the following. 

\begin{Lemma}\label{lem2}
$L'_{n}$ and $M'_{n}$ are ambient isotopic. 
\end{Lemma}

\begin{proof}
See Fig. \ref{Cnambi2}. 
\end{proof}

\begin{figure}[htbp]
      \begin{center}
\scalebox{0.375}{\includegraphics*{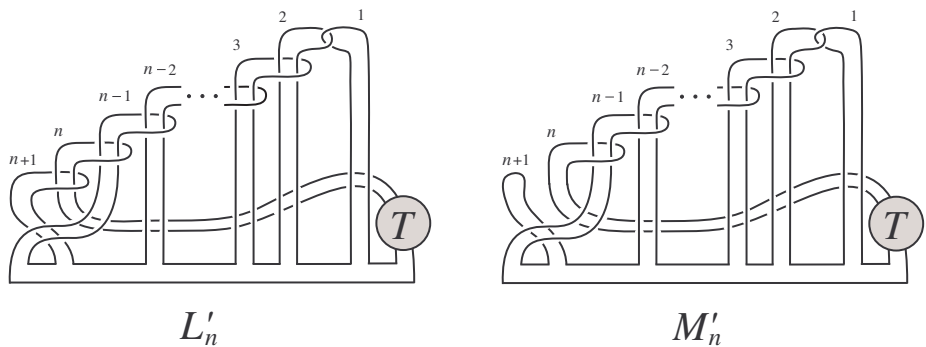}}
      \end{center}
   \caption{Two links $L'_{n}$ and $M'_{n}$ ($n\ge 3$)}
  \label{Cnambi}
\end{figure} 
\begin{figure}[htbp]
      \begin{center}
\scalebox{0.45}{\includegraphics*{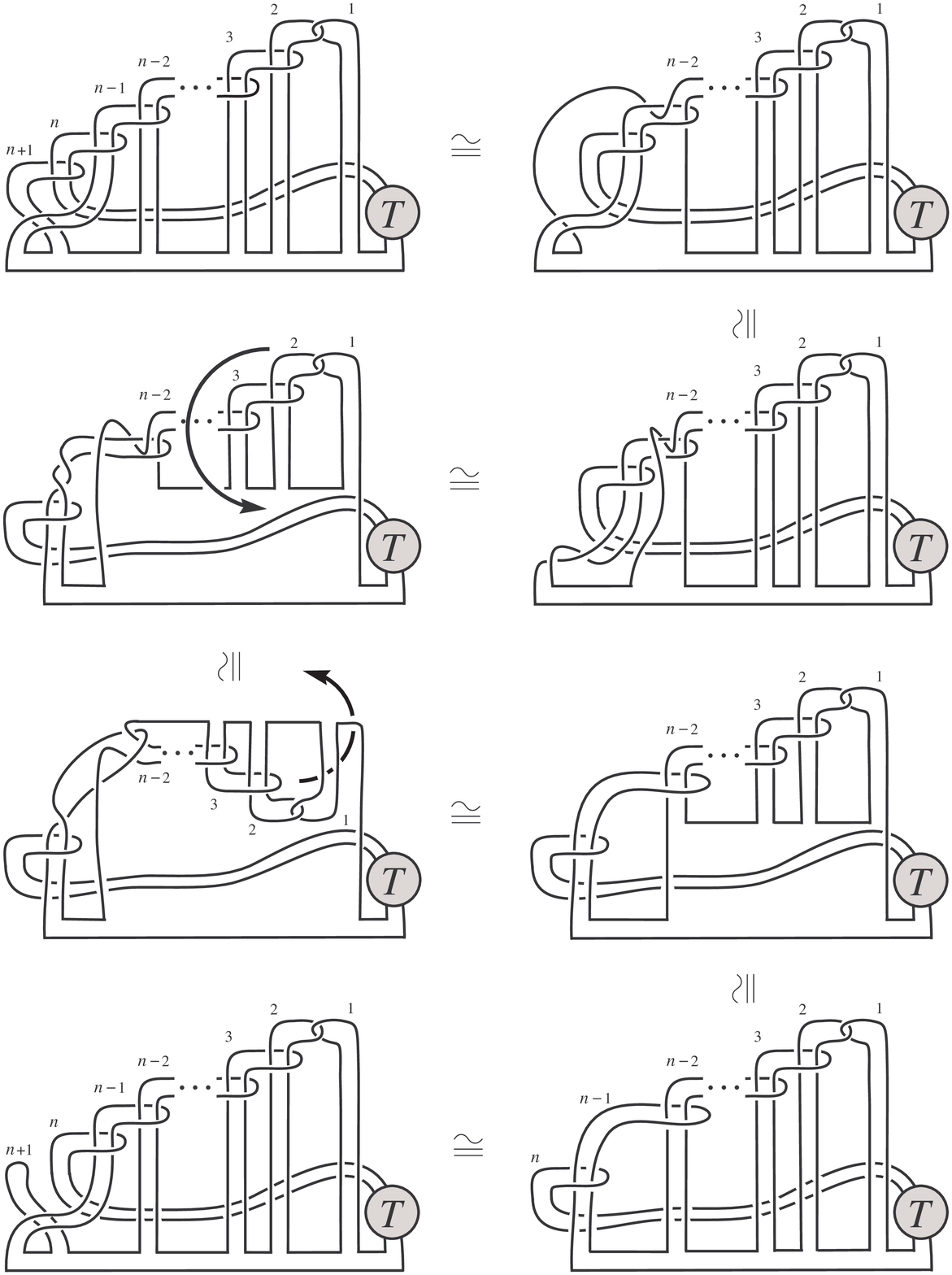}}
      \end{center}
   \caption{$L'_{n}$ and $M'_{n}$ are ambient isotopic}
  \label{Cnambi2}
\end{figure} 

Lemma \ref{lem2} gives a new example of a $C_{n}$-move which does not change the knot type. Such an example was first discovered by Ohyama-Tsukamoto \cite{OT99}. 

The third lemma is a calculation of the Jones polynomial for a $\left(2,-m\right)$-torus knot or link $N_{m}$ for a non-negative integer $m$ as illustrated in Fig. \ref{2mtorus} (1). Note that such a calculation has been already known, see \cite[pp. 37]{Kauffman3}, \cite[Lemma 2.1]{Landvoy98} for example. However we state a formula and give a proof for reader's convenience. 

\begin{Lemma}\label{lem3}
\begin{eqnarray*}
V_{N_{m}}(t) = 
\left(-t^{-\frac{1}{2}}\right)^{m}\left(-t^{\frac{1}{2}}-t^{-\frac{1}{2}}\right)+ \frac{\left(-t^{\frac{1-3m}{2}}\right)\left\{1-\left(-t\right)^{m}\right\}}{1+t}. 
\end{eqnarray*}
\end{Lemma}

\begin{proof}
Note that $N_{0}$ is the trivial $2$-component link and $N_{1}$ is the trivial knot. Then we can check the formula directly for $m=0,1$. Assume that $m\ge 2$. Then we obtain the skein triple $\left(N_{m-2},N_{m},N_{m-1}\right)$ easily and therefore we have  
\begin{eqnarray}\label{v1}
t^{-1}V_{N_{m-2}}(t) - tV_{N_{m}}(t) 
= \left(t^{\frac{1}{2}}-t^{-\frac{1}{2}}\right)V_{N_{m-1}}(t). 
\end{eqnarray}
By (\ref{v1}), we have 
\begin{eqnarray}\label{v2}
V_{N_{m}}(t) + t^{-\frac{1}{2}}V_{N_{m-1}}(t) 
&=& t^{-\frac{3}{2}}\left\{V_{N_{m-1}}(t) + t^{-\frac{1}{2}}V_{N_{m-2}}(t)\right\}\\
&=& \left(t^{-\frac{3}{2}}\right)^{m-1}\left\{V_{N_{1}}(t) + t^{-\frac{1}{2}}V_{N_{0}}(t)\right\}\nonumber\\
&=& -t^{\frac{1-3m}{2}}. \nonumber
\end{eqnarray}
Then by (\ref{v2}), we have 
\begin{eqnarray*}
V_{N_{m}}(t) &=& -t^{-\frac{1}{2}}V_{N_{m-1}}(t) -t^{\frac{1-3m}{2}} \\
&=& \left(-t^{-\frac{1}{2}}\right)^{m}V_{N_{0}}(t) + \sum_{i=1}^{m}\left(-t^{\frac{1-3m}{2}}\right)(-t)^{i-1}\\
&=& \left(-t^{-\frac{1}{2}}\right)^{m}\left(-t^{\frac{1}{2}}-t^{-\frac{1}{2}}\right) + \frac{\left(-t^{\frac{1-3m}{2}}\right)\left\{1-\left(-t\right)^{m}\right\}}{1+t}.
\end{eqnarray*}
\end{proof}

\begin{proof}[Proof of Theorem \ref{main_jones}]
First we show in the case of $n=3$ and $4$. If $n=3$, by a calculation (with the help of \cite{kodamaknot}) we have 
\begin{eqnarray*}
V_{J_{3}}(t) &=& t^{-1}-2+4t-4t^{2}+5t^{3}-5t^{4}+3t^{5}-2t^{6}+t^{7},  \\
V_{K_{3}}(t) &=& t^{-1}-1+2t-2t^{2}+2t^{3}-2t^{4}+t^{5}. 
\end{eqnarray*}
Then we have 
\begin{eqnarray*}
V_{J_{3}}(t) - V_{K_{3}}(t) = (t-1)^{3}\left(t^{2}+t+1\right)\left(t^{2}+1\right). 
\end{eqnarray*}
If $n=4$, by a calculation we have 
\begin{eqnarray*}
V_{J_{4}}(t) &=& -t^{-2}+4t^{-1}-8+13t-15t^{2}+17t^{3}-16t^{4}+12t^{5}-8t^{6}+4t^{7}-t^{8},  \\
V_{K_{4}}(t) &=& -t^{-2}+4t^{-1}-7+10t-11t^{2}+12t^{3}-10t^{4}+7t^{5}-4t^{6}+t^{7}. 
\end{eqnarray*}
Then we have 
\begin{eqnarray*}
V_{J_{4}}(t) - V_{K_{4}}(t) = -(t-1)^{4}\left(t^{2}+t+1\right)\left(t^{2}+1\right). 
\end{eqnarray*}

From now on, we assume that $n\ge 5$. Since $\varepsilon_{i} = 1$ for any $i$, we have 
\begin{eqnarray}\label{j1}
 V_{J_{n}}(t) - V_{K_{n}}(t) 
&=& 
t^{\frac{n}{2}}\left(t-1\right)^{n}
\sum_{\delta_{2},\ldots,\delta_{n}=\pm 1}\left(\prod_{j=2}^{n}\delta_{j}\right)V_{J_{n}\left[\delta_{2},\ldots,\delta_{n}\right]}(t). 
\end{eqnarray}
If $\delta_{2} = -1$, we can see that ${J_{n}}\left[-1,\delta_{3},\ldots,\delta_{n-1},1\right]$ and ${J_{n}}\left[-1,\delta_{3},\ldots,\delta_{n-1},-1\right]$ are ambient isotopic, see Fig. \ref{Cnex3}. Thus by (\ref{j1}), we have 
\begin{eqnarray}\label{j2}
V_{J_{n}}(t) - V_{K_{n}}(t) 
&=& 
t^{\frac{n}{2}}\left(t-1\right)^{n}
\sum_{\delta_{3},\ldots,\delta_{n}=\pm 1}\left(\prod_{j=3}^{n}\delta_{j}\right)V_{J_{n}\left[1,\delta_{3},\ldots,\delta_{n}\right]}(t). 
\end{eqnarray}

\begin{figure}[htbp]
      \begin{center}
\scalebox{0.375}{\includegraphics*{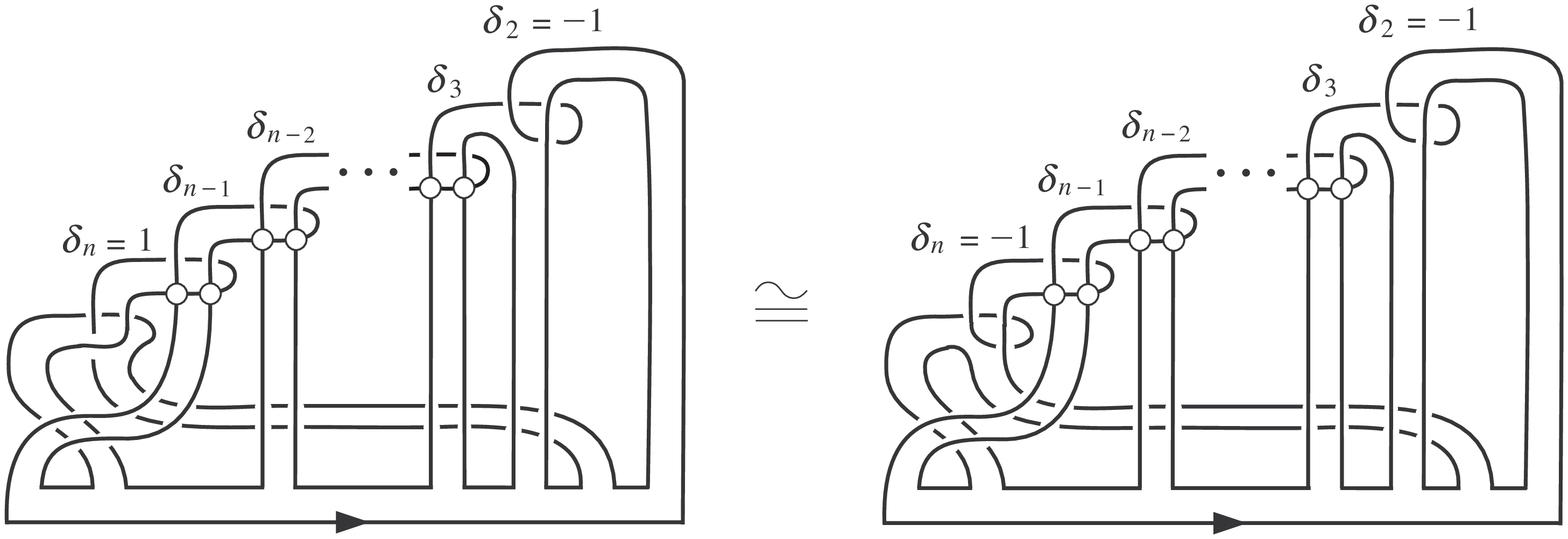}}
      \end{center}
   \caption{${J_{n}}\left[-1,\delta_{3},\ldots,\delta_{n-1},1\right]$ and ${J_{n}}\left[-1,\delta_{3},\ldots,\delta_{n-1},-1\right]$ are ambient isotopic}
  \label{Cnex3}
\end{figure} 

Let $k$ be an integer satisfying $3\le k\le n-2$. Note that $k$ is also satisfied with $3\le n-k+1\le n-2$. Then we can see that $J_{n}\left[1,\ldots,1,-1,\delta_{k+1},\ldots,\delta_{n}\right]$ is ambient isotopic to $L'_{n-k+1}\left[\delta_{k+1},\ldots,\delta_{n}\right]$ for some $2$-string tangle $T_{k}$, see Fig. \ref{Cnex4}, where $L'_{n-k+1}$ and $M'_{n-k+1}$ are corresponding knots as illustrated in Fig. \ref{Cnambi}. Then by Lemma \ref{lem2}, we have $L'_{n-k+1}$ and $M'_{n-k+1}$ are ambient isotopic and therefore 
\begin{eqnarray}\label{j3}
&& \sum_{\delta_{k+1},\ldots,\delta_{n}=\pm 1}\left(\prod_{j=k+1}^{n}\delta_{j}\right)V_{J_{n}\left[1,\ldots,1,-1,\delta_{k+1},\ldots,\delta_{n}\right]}(t)\\
&=& \left\{V_{L'_{n-k+1}}(t) - V_{M'_{n-k+1}}(t)\right\} \Big/ \left(-\prod_{i=k+1}^{n}\varepsilon_{i}\right)t^{-1+\sum_{i=k+1}^{n}\varepsilon_{i}-\frac{1}{2}(n-k)}\left(t-1\right)^{n-k} \nonumber\\
&=& 0. \nonumber
\end{eqnarray}
Thus by (\ref{j2}) and (\ref{j3}), we have 
\begin{eqnarray}\label{j4}
V_{J_{n}}(t) - V_{K_{n}}(t) &=& 
t^{\frac{n}{2}}\left(t-1\right)^{n}
\sum_{\delta_{n-1},\delta_{n}=\pm 1}\delta_{n-1}\delta_{n}
V_{J_{n}\left[1,\ldots,1,\delta_{n-1},\delta_{n}\right]}(t).
\end{eqnarray}

\begin{figure}[htbp]
      \begin{center}
\scalebox{0.35}{\includegraphics*{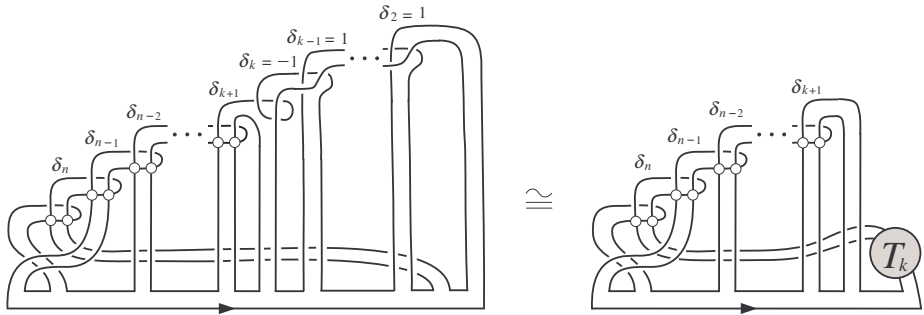}}
      \end{center}
   \caption{$J_{n}\left[1,\ldots,1,-1,\delta_{k+1},\ldots,\delta_{n}\right]$ is ambient isotopic to $L'_{n-k+1}\left[\delta_{k+1},\ldots,\delta_{n}\right]$ for some $2$-string tangle $T_{k}$}
  \label{Cnex4}
\end{figure} 

We can see easily that $J_{n}\left[1,\ldots,1,-1,-1\right]$ is ambient isotopic to $N_{n-3}$, $J_{n}\left[1,\ldots,1,-1,1\right]$ is ambient isotopic to the split union of $N_{n-4}$ and the trivial knot, and $J_{n}\left[1,\ldots,1,-1\right]$ is ambient isotopic to the connected sum of $N_{n-3}$, the Hopf link with linking number $1$ and the Hopf link with linking number $-1$. Thus we have 
\begin{eqnarray}
V_{J_{n}\left[1,\ldots,1,-1,-1\right]}(t) &=& V_{N_{n-3}}(t), \label{j5}\\
V_{J_{n}\left[1,\ldots,1,-1,1\right]}(t) &=& \left(-t^{\frac{1}{2}}-t^{-\frac{1}{2}}\right)V_{N_{n-4}}(t), \label{j6} \\
V_{J_{n}\left[1,\ldots,1,-1\right]}(t) &=& \left(-t^{\frac{5}{2}}-t^{\frac{1}{2}}\right)\left(-t^{-\frac{5}{2}}-t^{-\frac{1}{2}}\right)V_{N_{n-3}}(t). \label{j7}
\end{eqnarray}
Further, $J_{n}\left[1,\ldots,1\right]$ is ambient isotopic to the oriented link as illustrated in Fig. \ref{2mtorus} (2), where $m=n-4$. We obtain the skein triple $\left(N_{n-3},J_{n}\left[1,\ldots,1\right],N_{n-4}\right)$ by changing and smoothing the marked crossing in Fig. \ref{2mtorus}. Thus we have 
\begin{eqnarray}\label{j8}
V_{J_{n}\left[1,\ldots,1\right]}(t) 
= t^{-2}V_{N_{n-3}}(t) - t^{-1}\left(t^{\frac{1}{2}}-t^{-\frac{1}{2}}\right)V_{N_{n-4}}(t). 
\end{eqnarray}
By combining with (\ref{j4}), (\ref{j5}), (\ref{j6}), (\ref{j7}) and (\ref{j8}), we have 
\begin{eqnarray}\label{j9}
&& V_{J_{n}}(t) - V_{K_{n}}(t) \\ 
&=& 
t^{\frac{n}{2}}\left(t-1\right)^{n}
\left\{
\left(-1-t^{2}\right)V_{N_{n-3}}(t) + \left(t^{\frac{1}{2}}+t^{-\frac{3}{2}}\right)V_{N_{n-4}}(t)
\right\} \nonumber \\
&=& t^{\frac{n}{2}}\left(t-1\right)^{n}\left(1+t^{2}\right)
\left\{
-V_{N_{n-3}}(t) + t^{-\frac{3}{2}}V_{N_{n-4}}(t)
\right\}. \nonumber
\end{eqnarray}
Here, by Lemma \ref{lem3}, we also have 
\begin{eqnarray}\label{j10}
&& -V_{N_{n-3}}(t) + t^{-\frac{3}{2}}V_{N_{n-4}}(t)\\
&=& -\left(-t^{-\frac{1}{2}}\right)^{n-3}\left(-t^{\frac{1}{2}}-t^{-\frac{1}{2}}\right)- \frac{\left(-t^{\frac{10-3n}{2}}\right)\left\{1-\left(-t\right)^{n-3}\right\}}{1+t} \nonumber \\
&& + t^{-\frac{3}{2}}\left(-t^{-\frac{1}{2}}\right)^{n-4}\left(-t^{\frac{1}{2}}-t^{-\frac{1}{2}}\right) + \frac{\left(-t^{\frac{10-3n}{2}}\right)\left\{1-\left(-t\right)^{n-4}\right\}}{1+t} \nonumber \\
&=& (-1)^{n-2}t^{\frac{3-n}{2}}\left(-t^{\frac{1}{2}}-t^{-\frac{1}{2}}\right)
+ (-1)^{n-4}t^{\frac{1-n}{2}}\left(-t^{\frac{1}{2}}-t^{-\frac{1}{2}}\right) 
+ (-1)^{n-4}t^{\frac{2-n}{2}}\nonumber \\
&=& (-1)^{n+1}\left(t^{\frac{4-n}{2}}+t^{\frac{2-n}{2}}+t^{-\frac{n}{2}}\right) \nonumber \\
&=& (-1)^{n+1}t^{-\frac{n}{2}}\left(t^{2}+t+1\right). \nonumber
\end{eqnarray}
By (\ref{j9}) and (\ref{j10}), we have the desired conclusion. 
\end{proof}

\begin{figure}[htbp]
      \begin{center}
\scalebox{0.4}{\includegraphics*{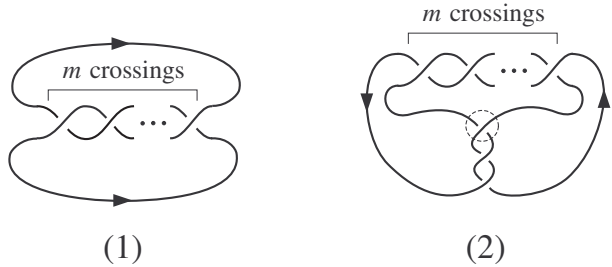}}
      \end{center}
   \caption{(1) $(2,-m)$-torus knot or link $N_{m}$, (2) A link ambient isotopic to $J_{n}\left[1,\ldots,1,1\right]$ if $m=n-4$}
  \label{2mtorus}
\end{figure} 

\section*{Acknowledgment}

The author is grateful to Professor Yoshiyuki Ohyama for informing him of Horiuchi's unpublished note \cite{Horiuchi}. 

{\normalsize
}

\end{document}